\documentclass[11pt,letterpaper]{article}

\usepackage{graphicx,float,latexsym,times}
\usepackage{amsfonts,amstext,amsmath,amssymb,amsthm}
\usepackage{mathrsfs}
\usepackage{bbm}

\usepackage{caption2}

\long\def\symbolfootnote[#1]#2{\begingroup%
\def\thefootnote{\fnsymbol{footnote}}\footnote[#1]{#2}\endgroup}
\usepackage{titlesec} % Very fancy section title control: Sets sections as SQA requires

\titleformat{\section}{\large\bfseries}{\thesection.}{.5em}{}
\titlespacing*{\section}{0pt}{*3}{*2}
\titleformat{\subsection}{\normalfont\bfseries}{\thesubsection.}{.5em}{}
\titlespacing*{\subsection} {0pt}{*3}{*2}
\titleformat{\subsubsection}{\normalfont\bfseries}{\thesubsubsection.}{.5em}{}
\titlespacing*{\subsubsection} {0pt}{*3}{*2}

\theoremstyle{plain} %% italic text

\newtheorem{lemma}{Lemma}[section]

\renewcommand{\refname}{{\large \bf REFERENCES}}
\numberwithin{equation}{section} %% double numbering within sections

\usepackage[authoryear]{natbib}
\newcommand{\ind}[1]{\mathbbm{1}_{\{#1\}}}   %indicator
\newcommand{\cF}{\mathscr{F}}
\newcommand{\cD}{\mathscr{D}}
\newcommand{\Real}{\mathbb{R}}
\newcommand{\Exp}{\mathsf{E}}
\newcommand{\Pro}{\mathsf{P}}

\begin{document}

\title{\textbf{\Large Minimax Optimality of Shiryaev-Roberts Procedure for Quickest Drift Change Detection of a Brownian motion}}

\date{}

\maketitle
\vspace{-2cm}
%%%%%%%%% Authors, affiliations %%%%%%%%%%%%%%%%%%%%%%%%%%

\author{
\begin{center}
%\vskip -1cm

\textbf{\large Taposh Banerjee and George V. Moustakides}
\end{center}
}

\symbolfootnote[0]{\normalsize Address correspondence to T. Banerjee, John A. Paulson School of Engineering and Applied Sciences, Harvard University, Cambridge, USA; E-mail: tbanerjee@seas.harvard.edu and G. V. Moustakides, Department of Electrical and Computer Engineering, University of Patras, 26500 Patras-Rion, Greece; E-mail: moustaki@upatras.gr}

{\small \noindent\textbf{Abstract:} 
The problem of detecting a change in the drift of a Brownian motion is considered. The change 
point is assumed to have a modified exponential prior distribution with unknown parameters. 
A worst-case analysis with respect to these parameters is adopted leading to a min-max problem formulation. Analytical and numerical justifications are provided towards establishing that the Shiryaev-Roberts procedure with a specially designed starting point is exactly optimal 
for the proposed mathematical setup. 
  
%Additional decentralized and
%centralized control based algorithms are also proposed. The
%performances of all the proposed algorithms are compared using
%simulations.
}
%%%%%%%%% Key words %%%%%%%%%%%%%%%%%%%%%%%%%%
\vspace{0.5cm}
\\ \\
%%%%%%%%% Subject Classifications %%%%%%%%%

%\vspace{-0.2cm}
\section{INTRODUCTION}\label{sec:Intro}
Consider a continuous-time stochastic process $\{\xi_t\}$ of the form
\begin{equation}
\label{eq:Obs}
d\xi_t = \mu dt\ind{t\geq\tau}+ dw_t, \quad t \geq 0,~~\xi_0=0,~~\mu\neq0
\end{equation}
where $\{w_t\}$ is a Wiener process and $\tau$ is a real valued random variable independent of $\{w_t\}$. 
The variable $\tau$ is to be interpreted as the change-point at which there is a change in the 
drift: for $t \leq \tau$, $\xi_t$ is a standard Brownian motion, 
and for $t > \tau$, it is a Brownian motion with a known drift $\mu \in \Real$.  
We consider the problem of detecting this change using a stopping time $T$ adapted to the filtration generated by the process $\{\xi_t\}_{t\geq0}$, 
with minimum possible delay $T-\tau$, subject to a constraint on false alarms $\{T \leq \tau\}$. To simplify our presentation, from now on without loss of generality we assume that $\mu=\sqrt{2}$. Indeed any other value of $\mu\neq0$ can be reduced to $\sqrt{2}$ just by a simple change in time scale of the process $\{\xi_t\}$ and a change in sign if $\mu<0$.

When the random variable $\tau$ with $\tau \in \Real$ has a zero-modified exponential prior
then the Bayesian version of the quickest change detection 
problem was studied by \cite{shir-siamtpa-1963} and the corresponding optimum test is known as the ``Shiryaev test''. In this work we are interested in the case where the parameters of the zero-modified exponential prior are \textit{unknown} and we follow a worst-case analysis to cope with this lack of information. Our claim is that the Shiryaev-Roberts procedure with a specially designed deterministic starting point, known as the SR-$r$ procedure (see \cite{Moustakides}), is exactly optimal for the proposed formulation. In fact we provide analytical and numerical evidence to support this claim. 

We now state the problem formulation and the main results of this paper rigorously. The observation process $\{\xi_t\}$ is as in \eqref{eq:Obs}. 
The zero-modified exponential prior for the change-point $\tau$ is such that
\begin{equation}\label{eq:prior_modexp0}
\Pro(\tau \leq 0) = \pi,~~\Pro(\tau \in dt) = (1-\pi)\lambda e^{-\lambda t}dt, ~~ t \geq 0,   
\end{equation}
for some $\lambda \geq 0$ and $\pi \in [0,1]$. The assumption of a zero-modified exponential prior is fundamental to our work, and will play a crucial role in what follows. However, for analytical convenience, it is necessary to change the corresponding parametrization. In particular we define $r=\frac{\pi}{\lambda}$ suggesting
\begin{equation}\label{eq:prior_modexp}
\Pro(\tau \leq 0) = r\lambda,~~\Pro(\tau \in dt) = (1-r\lambda)\lambda e^{-\lambda t} dt, ~~ t \geq 0,   
\end{equation}
with $\lambda\geq0$ and $\frac{1}{\lambda}\geq r\geq0$.

Regarding probability measures, we use $\Pro_t$ to denote the measure incurred, when the change-time takes upon the deterministic value $\tau=t$ and reserve $\Exp_t$ for the corresponding expectation. With this definition we have that $\Pro_\infty$ corresponds to the probability measure when all observations are under the nominal regime while $\Pro_0$ when the observations are under the alternative. Combining the previous measures with the prior on $\tau$ produces $\Pro_{r, \lambda}$ and $\Exp_{r, \lambda}$, that is, the probability measure and expectation when the change-time $\tau$ is random. 

If $\{\cF_t\}_{t\geq0}$ denotes the filtration generated by the observations, i.e.~$\cF_t=\sigma(\xi_s:0\leq s \leq t)$, with $\cF_0$ the trivial sigma-algebra then, 
for detection we seek an $\{\cF_t\}$-adapted stopping time $T$ that will detect the change in the drift as quickly as possible, subject to a constraint on the false alarm rate. 
When the pair $(r, \lambda)$ is known, \cite{shir-siamtpa-1963} proposed the following formulation\footnote{Actually to be more precise Shiryaev proposed a Bayesian version of the problem. It can be easily shown that \eqref{eq:ShirProb} can be reduced to it.}
\begin{equation}
\label{eq:ShirProb}
\inf_{T} \Exp_{r, \lambda}[T-\tau^+| T > \tau],~~~
\text{subject to:}~\Pro_{r, \lambda}(T \leq \tau) \leq \alpha,
\end{equation}
where $x^+= \max\{x, 0\}$ and $\alpha \in [0,1]$ a known false alarm probability level.

In the current work, unlike \eqref{eq:ShirProb}, we consider $(r,\lambda)$ to be \textit{unknown}. In order to deal with this lack of information we adopt a worst-case analysis with respect to the parameter pair. We therefore propose the following min-max constrained optimization alternative
\begin{equation}
\label{eq:minimaxProb}
\inf_{T}  \sup_{r, \lambda} \Exp_{r, \lambda}[T-\tau^+ | T > \tau],~~~
\text{subject to:}~\Exp_{\infty}[T] \geq \gamma,
\end{equation}
where $\gamma$ is a constant that constrains the average period of false alarms. The switching from the false alarm probability appearing in \eqref{eq:ShirProb}, to the average false alarm period adopted in \eqref{eq:minimaxProb} is common for min-max approaches (e.g. see \cite{Moustakides2}). This change is necessary since the false alarm probability in \eqref{eq:ShirProb} depends on the unknown parameter pair and would therefore require an additional worst-case analysis for the constraint. Unfortunately, the worst-case false alarm probability cannot be efficiently controlled (actually very often it takes the value 1) thus making the constrain meaningless. This is the reason why it is replaced by the average false alarm period which is independent from the unknown parameters.

To complete our introduction we need some additional definitions that are necessary for our analysis. Consider the process
\begin{equation}\label{eq:LR_u_t}
du_t = - dt + \sqrt{2} d\xi_t,~u_0=0,
\end{equation}
then from \cite{pesk-shir-book-2006}, Chapter VI, Section 22, and Girsanov's theorem (see \cite{rogersWillimans2000}) we have
$$
\frac{d\Pro_0}{d\Pro_\infty}(\cF_t)=e^{u_t},~t\geq0,
$$
and more generally for $s\geq0$
\begin{equation}
\label{eq:RadNikDeri}
\frac{d\Pro_s}{d\Pro_\infty}(\cF_t)=\left\{\begin{array}{cl}
e^{u_t-u_s}&s\leq t\\
1& s>t.
\end{array}
\right.
\end{equation}
It is clear that $e^{u_t}$ is the Radon-Nikodym derivative between the two probability measures $\Pro_0,\Pro_\infty$ limited to $\cF_t$ while $e^{u_t-u_s}$ is the Radon-Nikodym derivative between $\Pro_s,\Pro_\infty$ on the same sigma-algebra when $t\geq s$.

Consider now the following statistic which will play a key role in our analysis
\begin{equation}
R_t = e^{u_t} \left\{ r_* + \int_{0}^t e^{-u_s} ds\right\},
\label{eq:stat-R}
\end{equation}
where $r_*\geq0$ is a specially designed initial point (since $R_0=r_*$) which will be specified exactly in the sequel. Define now the following function
\begin{equation}
g(R)=e^{(r_*+\gamma)^{-1}}E_1\big((r_*+\gamma)^{-1}\big)-e^{R^{-1}}E_1\left(R^{-1}\right),
\label{eq:fun-g}
\end{equation}
where $E_1(x)=\int_x^\infty\frac{e^{-z}}{z}dz$ is the exponential integral\footnote{See \cite{Abramowitz}, Chapter 5.}, $r_*$ is the parameter we introduced in the definition of $R_t$ in \eqref{eq:stat-R} and $\gamma$ the constraint on the average false alarm period in \eqref{eq:minimaxProb}. The next lemma contains a number of interesting equalities which will be used throughout our analysis. The most important one consists in providing an alternative form for our performance measure.

\begin{lemma}\label{lem:1} If $R_t$ is as in \eqref{eq:stat-R} and $T$ an $\{\cF_t\}$-adapted stopping time, then we have the following equalities that are valid
\begin{align}
\Exp_\infty[R_T]&=r_*+\Exp_\infty[T]\label{eq:lem1.1-1}\\
\Exp_t[(T-t)^+|\cF_t]&=\Exp_t[g(R_t)-g(R_T)|\cF_t]\ind{T>t}.\label{eq:lem1.1-2}
\end{align}
Furthermore
\begin{multline}
\cD(T,r,\lambda)=\Exp_{r, \lambda}[T-\tau^+| T > \tau]\\
=\frac{r\Exp_0[g(r_*)-g(R_T)]+(1-\lambda r)\int_0^\infty\Exp_t\left[\big(g(R_t)-g(R_T)\big)\ind{T>t}\right] e^{-\lambda t}dt}{r+(1-\lambda r)\Exp_\infty\left[\int_0^T e^{-\lambda t}dt\right]}.
\label{eq:lem1.1-3}
\end{multline}
When $r=r_*$ and $\lambda=0$ then we can also write
\begin{equation}
\cD(T,r_*,0)=\frac{\Exp_\infty\left[\int_0^TR_tdt\right]}{r_*+\Exp_\infty[T]}.
\label{eq:lem1.1-4}
\end{equation}
\end{lemma}
\begin{proof}
The proof of this lemma is presented in the Appendix.
\end{proof}

\subsection{Saddle-Point Problem}
With the help of Lemma\,\ref{lem:1} the min-max problem depicted in \eqref{eq:minimaxProb} can be equivalently expressed as
\begin{equation}
\inf_{T}\sup_{r,\lambda}\cD(T,r,\lambda),~~~\text{subject to:}~\Exp_\infty[T]\geq\gamma.
\label{eq:min-max-final}
\end{equation}
As is the case in most min-max problems, it is possible to obtain their solution by solving a simpler \textit{saddle-point} alternative\footnote{\label{foot:2} \cite{Boyd}, Section 5.1: When a saddle-point solution exists it is also the solution of the min-max problem. The opposite is not necessarily true.}. In particular we are interested in a triplet $T_*,r_*,\lambda_*=0$ such that for any $\lambda\geq0$ and $\frac{1}{\lambda}\geq r\geq0$ we have validity of the following double inequality
\begin{equation}
\cD(T,r_*,0)\geq\cD(T_*,r_*,0)\geq\cD(T_*,r,\lambda),~~~\text{subject to:}~\Exp_\infty[T]\geq\gamma.
\label{eq:final-saddle}
\end{equation}
We should point out that with $\lambda_*=0$ the exponential prior becomes a degenerate uniform.

As we mention in Footnote \ref{foot:2}, it is a well established fact that the solution to the saddle-point problem in \eqref{eq:final-saddle} is also the solution to the min-max problem in \eqref{eq:min-max-final}. We therefore focus on \eqref{eq:final-saddle}.

\section{MAIN RESULTS}
Our first goal is to specify completely the triplet $T_*,r_*,\lambda_*$. So far we have that $\lambda_*=0$. Let us now define $T_*$ in terms of $r_*$. For this to be possible we focus on the first inequality of the saddle-point problem in \eqref{eq:final-saddle} which requires $\cD(T,r_*,0)\geq\cD(T_*,r_*,0)$ for all $T$ satisfying the constraint $\Exp_\infty[T]\geq\gamma$. In fact, we realize that $T_*$ must solve the following constrained minimization problem
\begin{equation}
\inf_T\cD(T,r_*,0)=\cD(T_*,r_*,0),~~~\text{subject to:}~\Exp_\infty[T]\geq\gamma.
\label{eq:lft-saddle-point}
\end{equation}
Minimizing $\cD(T,r_*,0)$ over $T$ is straightforward and the optimum stopping time is given in the next lemma.

\begin{lemma}\label{lem:1-1}
The stopping time that solves the constrained minimization problem depicted in \eqref{eq:lft-saddle-point} is given by
\begin{equation}
T_*=\inf\{t>0: R_t\geq\gamma+r_*\}.
\label{eq:opt-T}
\end{equation}
\end{lemma}

\begin{proof} To prove this lemma we use the expression for $\cD(T,r_*,0)$ provided in \eqref{eq:lem1.1-4}. We are interested in showing that among all $T$ that satisfy the false alarm constraint $\Exp_\infty[T]\geq\gamma$ the stopping time that solves the minimization
$$
\inf_T\frac{\Exp_\infty\left[\int_0^TR_tdt\right]}{r_*+\Exp_\infty[T]}
$$
is $T_*$ defined in \eqref{eq:opt-T}. This is a known result in discrete time (see \cite{Polunchenko2010}). The continuous time version follows a similar line of proof and uses classical optimal stopping arguments. The analysis presents no special difficulties, for this reason we do not provide any further details. We only point out that $T_*$ satisfies the constraint with equality. Indeed from \eqref{eq:lem1.1-1} and since $R_{T_*}=r_*+\gamma$ we have $r_*+\gamma=\Exp_\infty[R_{T_*}]=r_*+\Exp_\infty[T_*]$ from which we conclude that $\Exp_\infty[T_*]=\gamma$.
\end{proof}

The candidate stopping time $T_*$ is specified in terms of $r_*$ which is still unknown. To define $r_*$ we make use of the second inequality in the saddle-point problem \eqref{eq:final-saddle}, namely that
$\cD(T_*,r_*,0)\geq\cD(T_*,r,\lambda)$
for all $\lambda\geq0$ and $\frac{1}{\lambda}\geq r\geq0$. Since the second inequality in \eqref{eq:final-saddle} must be true for all $\lambda\geq0$ it must certainly be valid for $\lambda=0$. This implies that $r_*$ must be such that for any $r\geq0$ we have $\cD(T_*,r_*,0)\geq\cD(T_*,r,0)$. In other words $r_*$ must maximize $\cD(T_*,r,0)$ over $r$. In \eqref{eq:lem1.1-3} substituting $T=T_*$, $\lambda=0$, recalling that $R_{T_*}=r_*+\gamma$ and $g(R_{T_*})=g(r_*+\gamma)=0$, we can write
\begin{multline*}
\cD(T_*,r,0)
=\frac{rg(r_*)+\int_0^\infty\Exp_t\left[g(R_t)\ind{T_*>t}\right]dt}{r+\Exp_\infty[T_*]}\\
=\frac{rg(r_*)+\int_0^\infty\Exp_\infty\left[g(R_t)\ind{T_*>t}\right]dt}{r+\Exp_\infty[T_*]}
=\frac{rg(r_*)+\Exp_\infty\left[\int_0^{T_*}g(R_t)dt\right]}{r+\gamma}
\end{multline*}
where the second equality is due to the fact that $g(R_t)\ind{T_*>t}$ is $\cF_t$-measurable and on $\cF_t$ we know that $\Pro_t$ coincides with $\Pro_\infty$.
To maximize $\cD(T_*,r,0)$ over $r$, we observe in the last ratio both, that the numerator and the denominator are linear functions of $r$, therefore the ratio is maximized either for $r=0$ or $r=\infty$. In order for the maximum to be attained by any other value between these two extremes we need
\begin{equation}
g(r_*)=\frac{\Exp_\infty\left[\int_0^{T_*}g(R_t)dt\right]}{\gamma},~~\text{or equivalently}~~
\Exp_\infty\left[\int_0^{T_*}\big(g(R_t)-g(r_*)\big)dt\right]=0,
\label{eq:def-rstar}
\end{equation}
where for the last equation we used the fact that $\Exp_\infty[T_*]=\gamma$.
Condition \eqref{eq:def-rstar} is the equation through which we can compute $r_*$. Interestingly the same condition also assures that $\cD(T_*,r,0)=g(r_*)$ i.e.~that $\cD(T_*,r,0)$ is constant independent from $r$, namely, an \textit{equalizer} over $r$.

Summarizing: For the solution of the min-max problem in \eqref{eq:minimaxProb} we propose the candidate stopping time $T_*$ defined in \eqref{eq:opt-T}, where the parameter $r_*$ is obtained by solving \eqref{eq:def-rstar}. Regarding \eqref{eq:def-rstar}, in the next section we offer a more analytic expression.

\subsection{Optimality of the Proposed Test}
The optimality of our candidate stopping time is assured if we can show that the two inequalities in the saddle-point problem \eqref{eq:final-saddle} are true. We note that $T_*$ was constructed so that the first inequality is valid for all $T$ satisfying the false alarm constraint. Regarding the second inequality, by selecting $r_*$ through \eqref{eq:def-rstar} we guarantee $g(r_*)=\cD(T_*,r,0)$ for all $r\geq0$. However, for optimality we need to demonstrate the stronger version
\begin{equation}
g(r_*)\geq\cD(T_*,r,\lambda).
\label{eq:2nd-ineq-1}
\end{equation}
The next lemma presents a condition that can replace \eqref{eq:2nd-ineq-1} and it is easier to verify.

\begin{lemma}\label{lem:3}
The inequality in \eqref{eq:2nd-ineq-1} is equivalent to
\begin{equation}\label{eq:2nd-ineq-2}
\Exp_\infty\left[ \int_0^{T_*} e^{-\lambda t} \big(g(R_t) - g(r_*)\big) dt\right] \leq 0,~\forall \lambda\geq0.
\end{equation}
\end{lemma}
\begin{proof} The proof is simple. Replacing $T$ with $T_*$ in the definition of $\cD(T,r,\lambda)$ in \eqref{eq:lem1.1-3} and using the boundary condition $g(R_{T_*})=g(r_*+\gamma)=0$ we conclude that \eqref{eq:2nd-ineq-1} is true iff
$$
g(r_*)\geq\frac{rg(r_*)+(1-\lambda r)\Exp_\infty\left[\int_0^{T_*}e^{-\lambda t}g(R_t)dt\right]}{r+(1-\lambda r)\Exp_\infty\left[\int_0^{T_*}e^{-\lambda t}dt\right]},
$$
is valid for all $\lambda\geq0$ and $\frac{1}{\lambda}\geq r\geq0$. The above inequality is clearly equivalent to \eqref{eq:2nd-ineq-2} for $\frac{1}{\lambda}> r\geq0$, while it is trivially valid when $r=\frac{1}{\lambda}$.
\end{proof}

The next lemma provides a differential equation and suitable conditions for the computation of the left hand side expectation in \eqref{eq:2nd-ineq-2}.

\begin{lemma}\label{lem:4} Fix $\lambda\geq0$, if $f_{\lambda}(R)$ is a twice differentiable function of $R$ which is the solution of the ode
\begin{equation}
  - \lambda f_{\lambda}(R) + f_{\lambda}'(R) + R^2 f_{\lambda}''(R) = - \big(g(R) - g(r_*)\big)=
  e^{R^{-1}}E_1\left(R^{-1}\right)-e^{r_*^{-1}}E_1\left(r_*^{-1}\right),
\label{eq:funct-f}
\end{equation}
with $f_{\lambda}(R)$ bounded when $R\in[0,r_*+\gamma]$ and $f_\lambda(r_*+\gamma,\lambda)=0$, then
\begin{equation}
f_\lambda(r_*)=\Exp_\infty\left[ \int_0^{T_*} e^{-\lambda t} \big(g(R_t) - g(r_*)\big) dt\right].
\label{eq:fr}
\end{equation}
\end{lemma}
\begin{proof}
The proof is detailed in the Appendix.
\end{proof}

An analytic form for $f_0(R)$ (i.e. $f_\lambda(R)$ when $\lambda=0$), and how this function can be used in order to obtain an integral instead of a differential equation for $f_\lambda(R)$ when $\lambda>0$ is presented in the next lemma.

\begin{lemma}\label{lem:5} If $f_\lambda(R)$ is as in Lemma~\ref{lem:4} then for $\lambda=0$ the corresponding function $f_0(R)$ is equal to
\begin{equation}
f_0(R)=\{1-e^{r_*^{-1}}E_1(r_*^{-1})\}(R-r^*-\gamma)+\int_{(r^*+\gamma)^{-1}}^{R^{-1}}E_1(x)d\big(E_i(x)\big),
\label{eq:f0R}
\end{equation}
while $f_\lambda(R)$ when $\lambda>0$ satisfies the following integral equation
\begin{multline}
f_\lambda(R)=f_0(R)\\
-\lambda\Bigg\{
\left(E_i\big((r^*+\gamma)^{-1}\big)-e^{(r^*+\gamma)^{-1}}(r^*+\gamma)-E_i(R^{-1})+e^{R^{-1}}R\right)\int_{(r^*+\gamma)^{-1}}^{\infty}f_\lambda(z^{-1})d(e^{-z})\\
-\int_{(r^*+\gamma)^{-1}}^{R^{-1}}f_\lambda(z^{-1})d\big(e^{-z}E_i(z)\big)
+\left(E_i(R^{-1})-e^{R^{-1}}R\right)
\int_{(r^*+\gamma)^{-1}}^{R^{-1}} f_\lambda(z^{-1})d(e^{-z})
\Bigg\}
\label{eq:flambda}
\end{multline}
where $E_i(x)=\int_{-\infty}^x\frac{e^{z}}{z}dz$ is the second version of the exponential integral\footnote{See \cite{Abramowitz}, Chapter 5.}.
\end{lemma}
\begin{proof}
The details of the proof are presented in the Appendix.
\end{proof}

The function $f_0(R)$ enjoys an additional notable property. Comparing \eqref{eq:fr} with \eqref{eq:def-rstar} we observe that we can recover the expectation in \eqref{eq:def-rstar} by computing $f_0(r_*)$. This suggests that the corresponding equation can be written as $f_0(r_*)=0$. Using \eqref{eq:f0R} and substituting $R=r_*$ we obtain the final form of the equation which identifies $r_*$ and replaces \eqref{eq:def-rstar}
\begin{equation}
f_0(r_*)=-\gamma\{1-e^{r_*^{-1}}E_1(r_*^{-1})\}+\int_{(r_*+\gamma)^{-1}}^{r_*^{-1}}E_1(x)d\big(E_i(x)\big)=0.
\label{eq:det-rstar-final}
\end{equation}

To complete the proof of optimality for $T_*$ we need to establish the validity of \eqref{eq:2nd-ineq-2} which, because of $\eqref{eq:fr}$, is equivalent to showing that
\begin{equation}
f_\lambda(r_*)\leq0
\label{eq:opt-cond-final}
\end{equation}
where $f_\lambda(R)$ satisfies the integral equation in \eqref{eq:flambda}. Unfortunately this last step \textit{was not possible to demonstrate analytically}. Therefore we state the following claim:
\vskip0.3cm
\noindent\textbf{Conjecture.} \textit{The inequality $f_\lambda(r_*)\leq0$ is true for all $\lambda\geq0$.}
\vskip0.3cm

\noindent The validity of this claim establishes exact optimality of the candidate stopping time $T_*$ defined in \eqref{eq:opt-T} in the sense that it is min-max optimum according to the problem proposed in \eqref{eq:min-max-final}. Of course our conjecture constitutes a crucial part of the optimality proof for $T_*$. Even though we cannot support our claim analytically we intend to provide \textit{numerical evidence} for its validity by directly computing $f_\lambda(r_*)$ and examining its sign. To achieve this goal we develop a simple computational method by borrowing ideas from \cite{Moustakides}. In fact, as we will see next, the expressions for $f_0(R)$ and $f_\lambda(R)$ proposed in \eqref{eq:f0R} and \eqref{eq:flambda} respectively are properly set for the numerical computation of the two functions.

\subsection{Numerical Method}
For evaluating numerically $f_0(R)$ and $f_\lambda(R)$ we need to compute integrals of the form $\int_{\alpha}^{\beta} a(x)d\big(b(x)\big)$ where $a(x),b(x)$ are function of $x$ and $d\big(b(x)\big)=b'(x)dx$ denotes the differential of $b(x)$. If we sample the interval $[\alpha,\beta]$ (not necessarily canonically) at the points $\alpha=x_0<x_1<\cdots<x_N=\beta$ then using the simple trapezoidal rule we can approximate the corresponding integral by the following sum
\begin{multline}
\int_{\alpha}^{\beta} a(x)d\big(b(x)\big)\approx\sum_{n=1}^N\frac{a(x_n)+a(x_{n-1})}{2}\big(b(x_n)-b(x_{n-1})\big)\\
=\frac{b(x_1)-b(x_0)}{2} a(x_0)+\sum_{n=1}^{N-1} \frac{b(x_{n+1})-b(x_{n-1})}{2} a(x_n)+\frac{b(x_N)-b(x_{N-1})}{2} a(x_N).
\label{eq:approx-int}
\end{multline}
The last sum in \eqref{eq:approx-int} can be clearly written as the inner product $\mathbf{b}^t\mathbf{a}$ of the two vectors
\begin{align*}
\mathbf{a}&=[a(x_0),a(x_1),\ldots,a(x_N)]^t\\
\mathbf{b}&=\frac{1}{2}[b(x_1)-b(x_0),b(x_2)-b(x_0),\ldots,b(x_N)-b(x_{N-2}),b(x_N)-b(x_{N-1})]^t.
\end{align*}
This straightforward idea can be applied in \eqref{eq:det-rstar-final} for the computation of the corresponding integral and the evaluation of the function $f_0(r_*)$ for any given $r_*$. Furthermore, with the help of an elementary bisection method we can then easily approximate the root of the equation $f_0(r_*)=0$ and obtain the initializing point $R_0=r_*$ of our test statistic $R_t$.

Once $r_*$ is specified we can attempt to solve the integral equation \eqref{eq:flambda} in order to compute the function\footnote{It is more convenient to compute $f_\lambda(R^{-1})$ since it is the actual function used in the corresponding integrals.} $f_\lambda(R)$. We first sample the interval\footnote{We must avoid the value $R=0$ because it is the source of numerical instability. We can instead select a point which is sufficiently close to 0 but does not lead to the product of a very large with a very small number (which is the source of the observed instability).} $(0,r_*+\gamma]$ at a sufficient number of points.  Among our sampled values we must include $r_*$ since we are interested in (the sign of) $f_\lambda(r_*)$. Call $\mathbf{f}_\lambda$ the vector version of the samples of $f_\lambda(R)$ and $\mathbf{f}_0$ the corresponding vector for the samples of $f_0(R)$. In the sampled version of the integral equation \eqref{eq:flambda} if we approximate the three integrals using the idea proposed in \eqref{eq:approx-int} we end up with the following system of linear equations
$$
\mathbf{f}_\lambda=\mathbf{f}_0-\lambda\mathbf{P}\mathbf{f}_\lambda.
$$
Matrix $\mathbf{P}$ summarizes the contribution of the three integrals which use the function $f_\lambda(R)$. The reason we need a matrix (and not a vector) is because we evaluate \eqref{eq:flambda} for the complete collection of samples of $R$. Each sample requires its own vector $\mathbf{b}$ which contributes a row to the matrix $\mathbf{P}$. It is clear that the product $\mathbf{P}\mathbf{f}_\lambda$ evaluates the sum of the three integrals for all sampled values of $R$ at the same time. Solving for $\mathbf{f}_\lambda$ yields
\begin{equation}
\mathbf{f}_\lambda=(\mathbf{I}+\lambda\mathbf{P})^{-1}\mathbf{f}_0.
\label{eq:system-f}
\end{equation}
From the solution vector $\mathbf{f}_\lambda$ we only need to retain the term corresponding to $f_\lambda(r_*)$. We note that $\mathbf{f}_0,\mathbf{P}$ must be computed only once, since they do not depend on $\lambda$. By changing the value of the scalar $\lambda$ we can then find $f_\lambda(r_*)$ for different values of this parameter and examine its sign to verify the validity of \eqref{eq:opt-cond-final}.

\section{EXAMPLES}
Let us apply the numerical method we introduced above to the case where the average false alarm period takes the values $\gamma=5\text{ and }20$. The next two figures depict our numerical results. 

In Figure~\ref{fig:1}(a) we plot $f_0(r_*)$ as a function of $r_*$ for $\gamma=5$. For the computation of the integral in \eqref{eq:det-rstar-final} we used 501 samples in the interval $[r_*,r_*+\gamma]$. The bisection method estimated the root of $f_0(r_*)=0$ to be $r_*=1.0707$. This is the value we adopted for this parameter. For the computation of $f_\lambda(R)$ we sampled the interval $[2\times10^{-3},r_*+\gamma]$ at 2001 points retaining the 501 we used for the determination of $r_*$. We then solved the linear system in \eqref{eq:system-f} for 100 values of $\lambda$ selected canonically from the interval $(0,10]$. The resulting $f_\lambda(r_*)$ appears in Figure~\ref{fig:1}(b). We can see that this function is clearly negative thus supporting our conjecture.

In Figure~2(a),(b) we present our numerical results for the false alarm value $\gamma=20$. Here the bisection method yielded $r_*=1.5240$. For the computation of $f_0(r_*)$ and $f_\lambda(R)$, we used 1001 and 4001 samples respectively where for the latter case, as before, we sampled the interval $[2\times10^{-3},r_*+\gamma]$. Finally we selected canonically 200 samples for $\lambda$ from the interval $(0,10]$. The resulting function $f_\lambda(r_*)$ is depicted in Figure~2(b) and, as we can see, it is again negative thus supporting, once more, our claim.
\begin{figure}[h!]
\centerline{\includegraphics{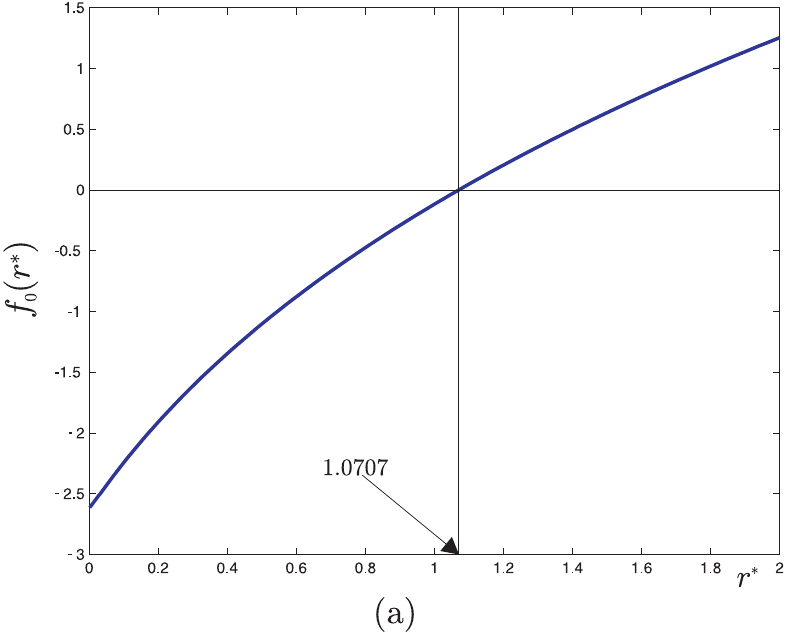}\hfill\includegraphics{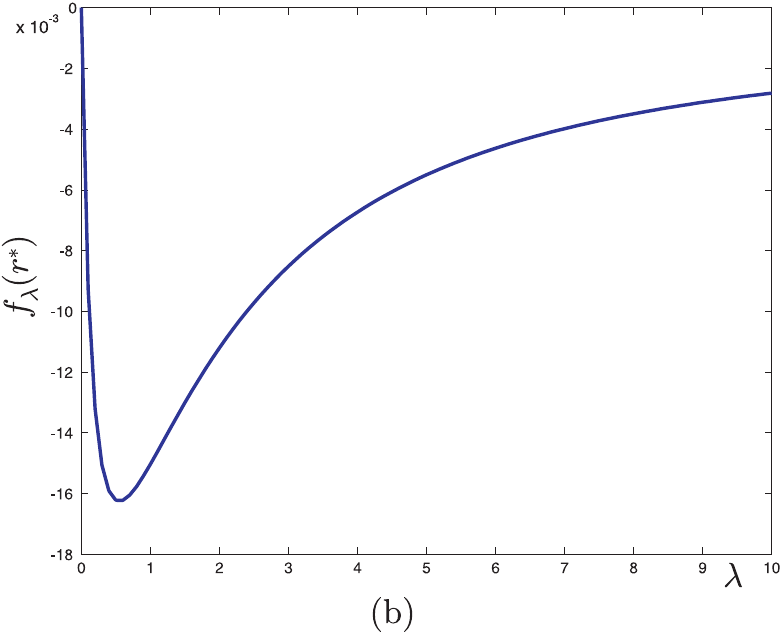}}
\caption{(a) Plot of $f_0(r_*)$ as a function of $r_*$ for $\mu = \sqrt{2}$ and $\gamma=5$. The point at which the function becomes 0 is $r_*=1.0707$. (b) Plot of $f_\lambda(r_*)$ as a function of $\lambda\geq0$.}
\label{fig:1}
\end{figure}
\begin{figure}[h!]
\centerline{\includegraphics{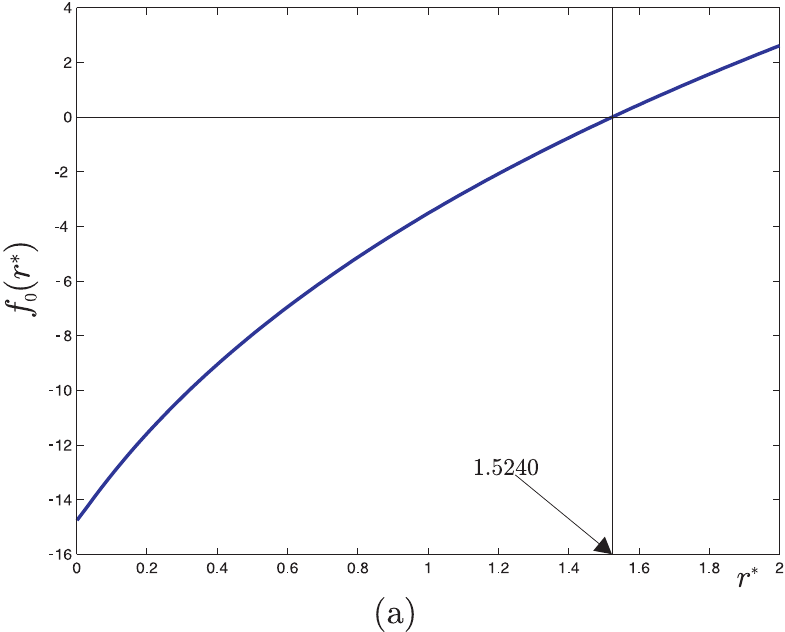}\hfill\includegraphics{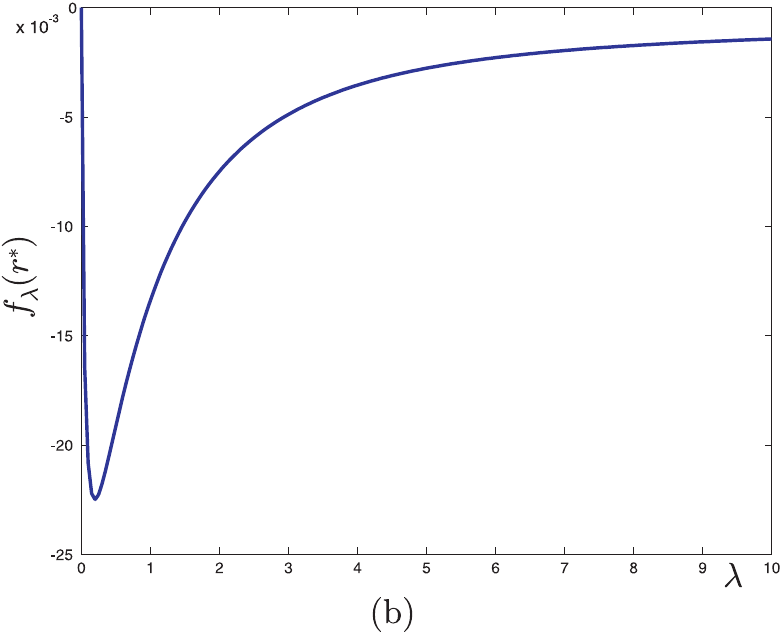}}
\caption{(a) Plot of $f_0(r_*)$ as a function of $r_*$ for $\mu = \sqrt{2}$ and $\gamma=20$. The point at which the function becomes 0 is $r_*=1.5240$. (b) Plot of $f_\lambda(r_*)$ as a function of $\lambda\geq0$.}
\label{fig:2}
\end{figure}

We should mention that we have performed numerous similar computations for various $\gamma$ that ranged from small to large values. In all cases $f_\lambda(r_*)$ turned out to be a negative function of $\lambda$. Of course it is understood that these observations cannot serve, by any means, as a formal proof of optimality for $T_*$. However, finding the proper formulas for the numerical computation demanded a serious mathematical analysis and the final outcome, undeniably, supports our conjecture and the optimality of our detector.

\section{DISCUSSION}
We must point out that the proposed stopping time $T_*$ in \eqref{eq:opt-T} is known as the Shiryaev-Roberts-$r$ (SR-$r$) test and has already been considered in the analysis of Pollak's performance measure, \cite{Pollak}
\begin{equation}
J_{\text{P}}(T)=\sup_{t\geq0}\Exp_t[T-t|T>t].
\label{eq:Pmetric}
\end{equation}
In the same article the following constrained min-max optimization problem was suggested
\begin{equation}
\inf_{T}J_{\text{P}}(T)=\inf_{T}\sup_{t\geq0}\Exp_t[T-t|T>t],~~\text{subject to:}~ \Exp_\infty[T]\geq\gamma>0,
\label{eq:Pollak}
\end{equation}
for the determination of an optimum detection strategy. Pollak was able to prove that the discrete time analog of the stopping time $T_*$ in \eqref{eq:opt-T} is \textit{asymptotically optimum} in a very strong sense, provided that the deterministic starting point $R_0=r_*$ is replaced by a \textit{random variable} which follows the quasi-stationary distribution. More precisely he demonstrated that
$$
J_{\text{P}}(T_*)-\inf_{T}J_{\text{P}}(T)=o(1),~~\text{as}~\gamma\to\infty.
$$
This type of asymptotic solution is called third order\footnote{First order asymptotic optimality is when $J_{\text{P}}(T_*)/\inf_{T}J_{\text{P}}(T)\to1$ and second when $J_{\text{P}}(T_*)-\inf_{T}J_{\text{P}}(T)\leq C<\infty$ uniformly in $\gamma$.} and has the important characteristic that, although the quantities $J_{\text{P}}(T_*)$ and $\inf_{T}J_{\text{P}}(T)$ tend to infinity as $\gamma\to\infty$, their distance  tends to zero. A similar third order asymptotic optimality property was proven by \cite{Tartakovsky} for the SR-$r$ test, namely the analog of $T_*$ in discrete time, with a deterministic and specially designed initialization $r_*$.

In continuous time there exist similar optimality claims. Specifically \cite{Polunchenko} shows that $T_*$ can solve \eqref{eq:Pollak} in the third order sense, when $r_*$ is random and follows the quasi-stationary distribution. This is the continuous time analog of Pollak's result. To obtain the equivalent of \cite{Tartakovsky} conclusions, one must demonstrate that the $T_*$ in \eqref{eq:opt-T} can also enjoy third order asymptotic optimality with proper deterministic initialization. Regarding the initializing value $r_*$ of $R_t$ in continuous time, we can be very precise. Since we are under an asymptotic regime with $\gamma\to\infty$ if we refer to \eqref{eq:det-rstar-final}, divide by $\gamma$ and let $\gamma\to\infty$, we arrive at the equation
$$
1-e^{r_*^{-1}}E_1(r_*^{-1})=0,
$$
From which we compute $r_*=2.299812$. Consequently, the claim is that $T_*$ when initialized with $r_*=2.299812$, becomes a third order asymptotic solution of the min-max problem defined in \eqref{eq:Pollak}. Unfortunately the proof of this statement is still an open problem.

We note that the Pollak metric in \eqref{eq:Pmetric} does not rely on any prior distribution (for $\tau$). It turns out that we can recover this criterion by considering a generic performance measure of the form $\Exp[T-\tau^+|T>\tau]$ where the prior for $\tau$ is \textit{unknown}. If we follow a worst-case approach over all possible priors then, as it is reported in \cite{Moustakides1}, we recover the Pollak metric. For this general case as we mentioned above, when $R_t$ is initialized with $r_*=2.299812$, the conjecture is that $T_*$ is third order asymptotically optimum.

In our current work we limit ourselves to the family of priors generated by the two-parameter zero-modified exponential density. We assume lack of exact knowledge of these parameters and we follow a worst-case analysis \textit{with respect to the two unknowns}. Since we adopt a significantly smaller class of distributions for the change-time (compared to Pollak's metric) our optimality claim can become stronger: We conjecture \textit{exact} optimality for $T_*$ in the sense that it is the exact solution of the min-max constrained optimization problem proposed in \eqref{eq:minimaxProb}. For $T_*$ to enjoy this optimality property, the initialization parameter $r_*$ must depend on $\gamma$ through equation \eqref{eq:det-rstar-final}. Even though we do not provide a complete analytical proof, we do however supply strong numerical evidence supporting the validity of our conjecture.

\section{APPENDIX}
In all proofs that follow we denote the threshold $r_*+\gamma$ with $A$ in order to simplify our mathematical analysis and the corresponding manipulations.
\vskip0.2cm
\noindent\textbf{Proof of Lemma \ref{lem:1}:} To show \eqref{eq:lem1.1-1} we recall that $\{e^{u_t}\}$ is an $\{\cF_t\}$-martingale with respect to $\Pro_\infty$, consequently $\Exp_\infty[e^{u_t}|\cF_s]=e^{u_s}$ when $t\geq s$. This can be extended to stopping times using Optional Sampling in the sense that
$\Exp_\infty[e^{u_T}|\cF_s]=e^{u_s}$ on the event $\{T\geq s\}$. With the help of this observation we can write
\begin{multline*}
\Exp_\infty[R_T]=\Exp_\infty\left[e^{u_T}\left\{r_*+\int_0^T e^{-u_s}ds\right\}\right]\\
=r_*\Exp_\infty[e^{u_T}]+\int_0^\infty\Exp_\infty\big[ \Exp_\infty[e^{u_T-u_s}|\cF_s]\ind{T>t}\big]ds
=r_*+\int_0^\infty\Exp_\infty[ \ind{T>t}]ds=r_*+\Exp_\infty[T],
\end{multline*}
which proves the desired expression.

For \eqref{eq:lem1.1-2} we use It\^o calculus and observe that under the $\Pro_0$ measure we have the following sde for $R_t$
\begin{equation}
dR_t=\left(2R_t+1\right)dt+\sqrt{2} R_tdw_t,~R_0=r_*.
\label{eq:dR-0}
\end{equation}
while under $\Pro_\infty$ the sde becomes
\begin{equation}
dR_t=dt+\sqrt{2} R_tdw_t,~R_0=r_*.
\label{eq:dR-infty}
\end{equation}
Consider now $dg(R_t)$ under $\Pro_0$, we have
$$
dg(R_t)=\{(2R_t+1)g'(R_t)+R_t^2g''(R_t)\}dt+\sqrt{2} R_tg'(R_t)dw_t
$$
Integrating and taking expectation with respect to $\Pro_t$, since we consider $R_T$ for $\{T>t\}$ we are under the post-change regime, namely $\Pro_0$. This yields
\begin{multline*}
\Exp_t[g(R_T)-g(R_t)|\cF_t]\ind{T>t}=\Exp_t\left[\int_t^T\{(2R_t+1)g'(R_t)+R_t^2g''(R_t)\}dt|\cF_t\right]\ind{T>t}\\
=-\Exp_t[T-t|\cF_t]\ind{T>t}=-\Exp_t[(T-t)^+|\cF_t]
\end{multline*}
where we used the fact that $\{T>t\}$ is $\cF_t$-measurable. We note that
the second equality is true because, as we can verify, $g(R)$ defined in \eqref{eq:fun-g} is the solution of the ode $(2R+1)g'(R)+R^2g''(R)=-1$.

To prove \eqref{eq:lem1.1-3} we observe that\footnote{We note that $\{T>\tau\}=\{T>\tau^+\}$ because $T>0$.}
$$
\Exp_{r,\lambda}[T-\tau^+|T>\tau]=\frac{\Exp_{r,\lambda}[(T-\tau^+)^+]}{\Pro_{r,\lambda}(T>\tau)}.
$$
We consider the numerator and denominator separately. We start with the denominator for which we can write
\begin{multline*}
\Pro_{r,\lambda}(T>\tau)=\Pro_0(T>0)\Pro(\tau\leq0)+
\int_0^\infty\Exp_{r,\lambda}[\ind{T>t}\ind{\tau\in dt}]\\
=\pi+\int_0^\infty\Exp_t[\ind{T>t}]\Pro(\tau\in dt)
=\pi+\int_0^\infty\Exp_t[\ind{T>t}](1-\pi)\lambda e^{-\lambda t}dt.
\end{multline*}
We note that when the time of change is at $\tau=t$, since $\{T>t\}$ is $\cF_t$-measurable, it is a pre-change event. But on $\cF_t$ the probability measure $\Pro_t$ coincides with the nominal $\Pro_\infty$, therefore the previous formula can be modified as follows
\begin{multline}
\Pro_{r,\lambda}(T>\tau)=\pi+(1-\pi)\int_0^\infty\Exp_\infty[\ind{T>t}] \lambda e^{-\lambda t}dt\\
=\pi+(1-\pi)\Exp_\infty\left[\int_0^T\lambda e^{-\lambda t}dt\right]
=\lambda\left\{r+(1-\lambda r)\Exp_\infty\left[\int_0^T e^{-\lambda t}dt\right]\right\}.
\label{eq:denom-final}
\end{multline}

Following a similar line of reasoning for the numerator, we obtain
$$
\Exp_{r,\lambda}[(T-\tau^+)^+]=\pi\Exp_0[T]+(1-\pi)\int_0^\infty\Exp_t[(T-t)^+]\lambda e^{-\lambda t}dt.
$$
Replacing $\Exp_t[(T-t)^+]$ from \eqref{eq:lem1.1-2} yields
\begin{multline}
\Exp_{r,\lambda}[(T-\tau^+)^+]=\\
\lambda\left\{r\Exp_0[g(R_0)-g(R_T)]+(1-\lambda r)\int_0^\infty\Exp_t\left[\big(g(R_t)-g(R_T)\big)\ind{T>t}\right] e^{-\lambda t}dt
\right\}.
\label{eq:numer-final}
\end{multline}
Taking the ratio of the numerator expression \eqref{eq:numer-final} and the expression for the denominator  in \eqref{eq:denom-final} and also recalling that $R_0=r_*$ yields the desired equality. 

To prove the last equality of this lemma, we consider the denominator of $\cD(T,r_*,0)$, normalize it by $\lambda$ and then take the limit as $\lambda\to0$. As we can then see from \eqref{eq:denom-final}, the denominator becomes $r_*+\Exp_\infty[T]$. For the numerator we propose the following alternative way to express $\Exp_t[(T-t)^+]$ that avoids the use of the function $g(R)$
\begin{multline*}
\Exp_t[(T-t)^+]=\Exp_t\left[\int_{t}^\infty\ind{T>t}\ind{T>s}ds\right]
=\int_{t}^\infty\Exp_t[\ind{T>t}\ind{T>s}]ds\\
=\int_{t}^\infty\Exp_\infty\left[\Exp_t[\ind{T>s}|\cF_t]\ind{T>t}\right]ds
=\int_{t}^\infty\Exp_\infty\left[\Exp_\infty[e^{u_s-u_t}\ind{T>s}|\cF_t]\ind{T>t}\right]ds\\
=\int_{t}^\infty\Exp_\infty[e^{u_s-u_t}\ind{T>t}\ind{T>s}]ds
=\Exp_\infty\left[\ind{T>t}\int_{t}^T e^{u_s-u_t}ds\right],
\end{multline*}
where in the forth equality we used \eqref{eq:RadNikDeri} and the fact that $\{T>s\}$ is $\cF_s$-measurable. Normalizing the numerator by $\lambda$ then letting $\lambda\to0$ and using the previous expression, we obtain
\begin{multline}
r_*\Exp_0[T]+\int_0^\infty\Exp_t[(T-t)^+]e^{-\lambda t}dt=
r_*\Exp_\infty\left[\int_{0}^T e^{u_s}ds\right]+
\Exp_\infty\left[
\int_0^T\left(\int_t^T e^{u_s}ds\right)e^{-u_t}dt
\right]\\
=r_*\Exp_\infty\left[\int_{0}^T e^{u_s}ds\right]+
\Exp_\infty\left[
\int_0^T\left(\int_0^s e^{-u_t}dt\right) e^{u_s}ds
\right]\\
=\Exp_\infty\left[\int_{0}^T e^{u_s}\left\{r_* +
\int_0^se^{-u_t}dt\right\} ds
\right]=\Exp_\infty\left[\int_0^T R_sds\right].
\label{eq:numer-normal}
\end{multline}
Dividing the expression for the normalized numerator in \eqref{eq:numer-normal} with the expression for the normalized denominator $r_*+\Exp_\infty[T]$ yields the desired result.
This concludes the proof of the lemma.\qed
\vskip0.2cm

\noindent\textbf{Proof of Lemma \ref{lem:4}:} To prove this lemma we use methodology similar to the one applied in Lemma~\ref{lem:1}. Consider $f_{\lambda}(R)$ to be twice differentiable, then under $\Pro_\infty$ we have
$$
d\big(e^{-\lambda t}f_{\lambda}(R_t)\big)=e^{-\lambda t}\{-\lambda f_{\lambda}(R_t)+g'(R_t)+
R_t^2g''(R_t)\}dt+\sqrt{2}e^{-\lambda t} R_tg'(R_t)dw_t
$$
from which we conclude that
$$
\Exp_\infty[e^{-\lambda T_*}f_{\lambda}(R_{t_*})-f_{\lambda}(R_0)]=
\Exp_\infty\left[
\int_0^{T_*}e^{-\lambda t}\{-\lambda f_{\lambda}(R_t)+g'(R_t)+
R_t^2g''(R_t)\}dt
\right].
$$
We select $f_{\lambda}(R)$ to satisfy the ode
$$
-\lambda f_{\lambda}(R)+f_{\lambda}'(R)+R^2f_{\lambda}''(R)=-\big(g(R)-g(r_*)\big)
$$
and to be bounded in $[0,A]$ with the boundary condition $f(A)=0$. If we substitute in the previous equality, after recalling that $R_0=r_*$ and $R_{T_*}=A$, we prove the desired result.\qed
\vskip0.2cm

\noindent\textbf{Proof of Lemma \ref{lem:5}:} We have that the function $f_0(R)$ satisfies the ode
\begin{equation}
f_0'(R) + R^2 f_0''(R) = 
e^{R^{-1}}E_1(R^{-1})-e^{r_*^{-1}}E_1(r_*^{-1}),
\label{eq:ode-f0}
\end{equation}
and it is bounded for $R\in[0, A]$ with $f_0(A)=0$. With direct substitution we can verify that the desired solution has the following form
\begin{multline*}
f_0(R)=\{1-e^{r_*^{-1}}E_1(r_*^{-1})\}(R-A)+\int_{A^{-1}}^{R^{-1}}\frac{e^x}{x}E_1(x)dx\\
=\{1-e^{r_*^{-1}}E_1(r_*^{-1})\}(R-A)+\int_{A^{-1}}^{R^{-1}}E_1(x)d\big(E_i(x)\big)
\end{multline*}

We can apply similar ideas in the differential equation \eqref{eq:funct-f} that defines $f_\lambda(R)$. Multiplying both sides with $e^{-R^{-1}}\frac{1}{R^2}$ yields
$$
\big(e^{-R^{-1}}f_\lambda'(R)\big)'=-\lambda e^{-R^{-1}}\frac{1}{R^2}f_\lambda(R)+\big(e^{-R^{-1}}f_0'(R)\big)'
$$
where for the last term we used the ode in \eqref{eq:ode-f0} that defines $f_0(R)$. Integrating both sides, we obtain
\begin{multline*}
f_\lambda(R)=f_0(R)-\lambda\int_{R}^{A}e^{x^{-1}}\left(\int_0^x\frac{e^{-z^{-1}}}{z^2}f_\lambda(z)dz\right)dx\\
=f_0(R)-\lambda\int_{A^{-1}}^{R^{-1}}\frac{e^x}{x^2}\left(\int_x^\infty e^{-z}f_\lambda(z^{-1})dz\right)dx\\=f_0(R)-\lambda\Bigg\{
\int_{A^{-1}}^{R^{-1}}e^{-z}f_\lambda(z^{-1})\left(\int_{A^{-1}}^z\frac{e^x}{x^2}dx\right)dz
+\left(\int_{A^{-1}}^{R^{-1}}\frac{e^x}{x^2}dx\right)\int_{R^{-1}}^\infty e^{-z}f_\lambda(z^{-1})dz\Bigg\},
\end{multline*}
where the second equality is the result of applying the change of variables $x\to x^{-1}$ and $z\to z^{-1}$ and the third is obtained by changing the order of integration in the double integral combined with careful housekeeping of the integration regions. The next step is to observe that the indefinite integral of $\frac{e^x}{x^2}$ is equal to $E_i(x)-\frac{e^x}{x}$. This applied in the previous expression yields
\begin{multline*}
f_\lambda(R)=f_0(R)\\
-\lambda\Bigg\{
\left(E_i(A^{-1})-e^{A^{-1}}A\right)\int_{A^{-1}}^{R^{-1}}f_\lambda(z^{-1})d(e^{-z})
+\int_{A^{-1}}^{R^{-1}}f_\lambda(z^{-1})\big(e^{-z}E_i(z)-z^{-1}\big)dz\\
+\left(E_i(A^{-1})-e^{A^{-1}}A-E_i(R^{-1})+e^{R^{-1}}R\right)
\int_{R^{-1}}^\infty f_\lambda(z^{-1})d(e^{-z})
\Bigg\}.
\end{multline*}
Combining terms and observing that the indefinite integral of $e^{-z}E_i(z)-z^{-1}$ is $-e^{-z}E_i(z)$, yields
\begin{multline*}
f_\lambda(R)=f_0(R)
-\lambda\Bigg\{
\left(E_i(A^{-1})-e^{A^{-1}}A-E_i(R^{-1})+e^{R^{-1}}R\right)\int_{A^{-1}}^{\infty}f_\lambda(z^{-1})d(e^{-z})\\
-\int_{A^{-1}}^{R^{-1}}f_\lambda(z^{-1})d\big(e^{-z}E_i(z)\big)
+\left(E_i(R^{-1})-e^{R^{-1}}R\right)
\int_{A^{-1}}^{R^{-1}} f_\lambda(z^{-1})d(e^{-z})
\Bigg\},
\end{multline*}
which is the final expression. This concludes the proof of Lemma~\ref{lem:5}.\qed

\section*{ACKNOWLEDGEMENT}
This work was supported by the US National Science Foundation under Grant CIF\,1513373, through Rutgers University.

%\newpage

\end{document}